\documentclass[twoside,11pt]{article}

\usepackage{jmlr2e}
\usepackage{amsmath}
\usepackage{dsfont}

\def\a{\alpha}

\def\g{\gamma}

\def\d{\delta}

\def\e{\epsilon}

\def\t{\theta}

\def\ie{\textit{i.e.}, }

\def\cf{\textit{cf.}}

\def\RR{\mathbb R}

\def\fcar{\mathds{1}}

\def\suchthat{\,|\,}
	
\def\donc{\Rightarrow}

\def\esp{\mathbf E}
\def\var{\mathbf{Var}}

\def\prob{\mathbf P}
\def\calN{\mathcal N}
\def\simiid{\overset{iid}{\sim}}

\def\nzeroun{\mathcal{N}(0,1)}

\newcounter{assumption}[section]

\ShortHeadings{General additive functional estimation}{Collier and Comminges}
\firstpageno{1}

\begin{document}

\title{Minimax optimal estimators for general additive functional estimation}

\author{\name Olivier Collier \email olivier.collier@parisnanterre.fr \\
       \addr MODAL'X \\
       Universit\'e Paris-Nanterre\\
       Nanterre, 92000, France\\
       and\\
       CREST, ENSAE\\
       Palaiseau, 91120
       \AND
       \name La\"etitia Comminges \email comminges@ceremade.dauphine.fr \\
       \addr CEREMADE\\
       Universit\'e Paris-Dauphine\\
       Paris, 75775, France\\
       and\\
       CREST, ENSAE\\
       Palaiseau, 91120}

\editor{}

\maketitle

\begin{abstract}
	In this paper, we observe a sparse mean vector through Gaussian noise and we aim at estimating some additive functional of the mean in the minimax sense. More precisely, we generalize the results of \citep{CollierCommingesTsybakov2017,CollierCommingesTsybakov2019} to a very large class of functionals. The optimal minimax rate is shown to depend on the polynomial approximation rate of the marginal functional, and optimal estimators achieving this rate are built.
\end{abstract}


\section{Introduction}\label{section:introduction}

In the general problem of functional estimation, one is interested in estimating some quantity $\bF(\bt)$ where $\bt\in\bT$ is an unknown parameter and $\bF$ is a known function. Information on this quantity is provided by an observation $\by \sim \prob_\bt$, where $(\prob_\bt)_{\bt\in\bT}$ is some family of probability distributions.

An exhaustive bibliography on the subject of functional estimation is out of the scope of this paper, but typical examples include functionals of a density function, e.g. the integrals of its square \cite{BickelRitov1988}, of smooth functionals of its derivatives \cite{BirgeMassart1995} or of nonlinear functionals \cite{KerkyacharianPicard1996}.

In this work, we focus on the case where $\bt\in\RR^d$ is a finite vector and $\bF$ is an additive functional, \ie
\begin{equation}\label{additive:functional}
	\bF(\bt) = \sum_{i=1}^d F(\t_i),
\end{equation}
which has now been well studied in the literature. For example, in the case when $\prob_\bt$ is the multinomial distribution $\mathcal{M}(n, p_1,\ldots, p_d)$, Shannon's or Rényi's entropy, which correspond respectively to marginal functionals $F(t)=-t\log(t)$ and $F(t)=t^{\a}$, are considered in \cite{HanJiaoWeissman2015, JiaoVenkatHanWeissman2015,WuYang2016} among others. The distinct elements problem, \ie finding how many different colors are present among at most $d$ colored balls in an urn, can also be expressed in this form \cite{PolyanskiyWu2019,WuYang2018}. Moreover, the quadratic functional defined by $F(t)=t^2$ is key in the problem of signal detection \cite{CarpentierCollierCommingesTsybakovWang2018}, and when the vector $\bt$ is assumed to be sparse, \ie when most of its coefficients are assumed to be exactly $0$, it also plays a crucial role for noise variance estimation \cite{CommingesCollierNdaoudTsybakov2019}. Finally, robust estimation of the mean is shown in \cite{CollierDalalyan2019} to be related with a linear functional of the outliers.

Here, our aim is not to focus on some particular functional, but to exhibit optimal minimax rates over large classes of functionals. Furthermore, we consider the Gaussian mean model, \ie
\begin{equation}\label{Gaussian:mean:model}
	\by \sim \calN(\bt,I_d) \quad \donc \quad y_i = \t_i + \xi_i, \quad \xi_i \simiid \nzeroun
\end{equation}
and we measure the quality of an estimator by the minimax risk defined by
\begin{equation}\label{minimax:risk}
	\sup_{\bt\in \bT} \esp_{\bt}\big( \hat{\bF} - \bF(\bt) \big)^2,
\end{equation}
where $\bT$ is some set of parameters. This framework was also used in \cite{CollierCommingesTsybakov2017,CaiLow2011,CollierCommingesTsybakov2019}, where respectively the cases when $F(t)=t$ or $F(t)=t^2$, $F(t) = |t|$ and $F(t) = |t|^\g$ for $0<\g\le 1$ are studied. It is clear from the last two papers that for rapidly growing functionals, it is relevant to restrict the set of $\bt's$ to a bounded subset of~$\RR^d$. Therefore, we assume that each component of $\bt$ belongs to a
segment, which we take for simplicity sake in the form $[-M,M]$. Finally, we place ourselves in a sparse context, which means that we assume the number of nonzero coefficients of $\bt$ -- its $l_0$-norm -- to be bounded by a known quantity, and we define
\begin{equation}\label{set:Theta}
	\bT \triangleq \bT_{s,M} = \big\{ \bt\in\RR^d \suchthat \|\bt\|_0\le s, \|\bt\|_\infty \le M \big\}.
\end{equation}

In this paper, we build minimax rate-optimal estimators when the functional $F$ is not too regular in the sense of polynomial approximation and does not grow too fast, when $s$ is at least of the order of $\sqrt{d}$ and $M$ is at most of order $\sqrt{\log(s^2/d)}$, showing that the polynomial approximation based method developed in \cite{CollierCommingesTsybakov2019} can be extended to a very broad class of functionals. More precisely, we make the following assumptions, where we use the notation $\d_{K,M}$ that is introduced in~\eqref{def:delta} below:

\begin{enumerate}
	\item[\textbf{(A0)}] $F$ is continuous on $[-\sqrt{\log(s^2/d)},\sqrt{\log(s^2/d)}]$.
	\item[\textbf{(A0')}] $F$ is continuous on $[-\sqrt{\log(s)},\sqrt{\log(s)}]$.
	\item[\textbf{(A1)}] There exist positive real numbers $\e_1,C_1$ such that
	\begin{equation}
		\sup_{\sqrt{2\log(s^2/d)}\le M\le \sqrt{2\log(d)}} \frac{\|F-F(0)\|_{\infty,[-M,M]}}{e^{\e_1M^2}} \le C_1.
	\end{equation}
	\item[\textbf{(A2)}] There exist  positive real numbers $\e_2,C_2$ such that
	\begin{equation}
		\sup_{\sqrt{2\log(s^2/d)}\le M\le \sqrt{2\log(d)}} \frac{ \d^{-1}_{M^2,M} }{e^{\e_2M^2}} \le C_2.
	\end{equation}
	\item[\textbf{(A3)}] $\forall \a>0$, $\exists f_\a>0$ such that if $|1-K_1/K_2|\vee|1-M_1/M_2| \le \a$, then $$f_\a^{-1} \le \frac{\d_{K_1,M_1}}{\d_{K_2,M_2}} \le f_\a.$$
\end{enumerate}
We make the first assumption on the continuity of $F$ for simplicity sake. Indeed, it would not be hard to extend the result to the case of a functional that is piecewise continuous with a finite number of discontinuities, even if discontinuous functionals might not be very important in practice. The second assumption is very mild, since estimation of rapidly growing functionals leads to very large minimax rates, making such problems uninteresting in practice. However the third assumption is essential: it expresses how the polynomial approximation rate drives the quality of estimation of the associated additive functional. Assumption~(A2) thus requires that $F$ is not smooth enough to be very quickly approximated by polynomials. In Section~\ref{section:optimality}, we recall the relation between polynomial approximation and differentiability. Finally, the last assumption is convenient to show that our lower and upper bounds match up to a constant. We believe that it is satisfied for all reasonable functionals.

Our theorems allow to recover some of the results implied by \cite{CaiLow2011,CollierCommingesTsybakov2019}, but cover a large part of all possible functionals. Note that some papers have already tackled the problem of general functionals. In \cite{FukuchiSakuma2017}, the authors give optimal rates of convergence for additive functionals in the discrete distribution case, when the fourth-derivative of the marginal functional is close in sup-norm to an inverse power function. In \cite{KoltchinskiiZhilova2018}, the case of general, not necessarily additive, functionals is considered in the Gaussian mean model with arbitrary covariance matrix. However, their results differ significantly from ours since they consider minimax risk over all marginal functionals belonging to some relatively small set of bounded and smooth functions in the Hölder sense. For example, none of the results obtained in \cite{CollierCommingesTsybakov2017,CollierCommingesTsybakov2019,CaiLow2011} can be recovered. Finally, the minimax rate for even larger classes of functionals, under constraints in the form $\sum_{i=1}^d c(\t_i) \le 1$ which includes sparsity, is obtained in~\cite{PolyanskiyWu2019} in term of the quantity
\begin{equation}\label{def:chi2modulus}
	\sup_{\pi_1,\pi_2} \Big\{ \Big| \int F(\t)\,\pi_1(d\t) - \int F(\t)\,\pi_2(d\t) \Big| \suchthat \chi^2(\prob_{\pi_1},\prob_{\pi_2})\le \frac1d, \esp_{\pi_i} \sum_{i=1}^d c(\t_i) \le 1 \Big\}
\end{equation}
where $\prob_\pi = \int \prob_\t\,\pi(d\t)$, $\chi^2(\prob_{\pi_1},\prob_{\pi_2})$ is the chi-square divergence between probabilities $\prob_{\pi_1}$ and $\prob_{\pi_2}$ and the supremum is taken over all probability distributions on~$\bT$. Their theorems allow for example to recover the minimax rate from~\cite{CaiLow2011} when $\Theta$ is bounded,  and may also allow to get the minimax rates from this paper. However, they do not exhibit generic estimators achieving the minimax risk. This paper fills in this gap in some cases.

\subsection*{Organization of the paper}

In Section~\ref{section:estimateur}, we build rate-optimal estimators of the additive functional and assess their performance. In Section~\ref{section:optimality}, we prove their optimality up to constants, and discuss conditions under which Assumption~(A2) is satisfied. The proofs of the theorems are postponed to Section~\ref{section:proof:theorem}, while technical lemmas can be found in~Section~\ref{section:lemma}.

\section{Upper bounds}\label{section:estimateur}
	
\subsection{Polynomial approximation}

Here, we set the notation on polynomial approximation that will be used throughout this paper. First denote $\mathcal{P}_K$ the set of polynomials of degree at most $K$, then define the polynomial of best approximation of $F$ on $[a,b]$ by
\begin{equation}\label{def:PKM}
	P_{K,[a,b]} = \underset{{P\in\mathcal{P}_K}}{\operatorname{\arg\min}} \| F - P \|_{\infty,[a,b]}
\end{equation}
and the polynomial approximation rate by
\begin{equation}\label{def:delta}
	\d_{K,[a,b]} = \| F - P_{K,[a,b]} \|_{\infty,[a,b]}.
\end{equation}
In the following, we write $P_{K,M} = P_{K,[-M,M]}$, $\d_{K,M}=\d_{K,[-M,M]}$, and we decompose~$P_{K,M}$ in the canonical base as 
\begin{equation}\label{def:coefficients}
	P_{K,M} = \sum_{k=0}^K a_{k,K,M} X^k.
\end{equation}

\subsection{Definition of the estimator and main theorem}

First, we use the sample duplication trick to transform  observation $y_i$ into independent randomized observations $y_{1,i}, y_{2,i}$ while keeping the same mean. Let us consider random variables $z_1,\ldots,z_d\simiid \nzeroun$ independent of $\by$ and define
\begin{equation}
	y_{1,i} = y_i + z_i, \quad y_{2,i} = y_i - z_i,
\end{equation}
so that $y_{1,i}, y_{2,i}\simiid\calN(\t_i,2)$. Yet for convenience, we will assume that $y_{1,i}, y_{2,i}\simiid\calN(\t_i,1)$.

Then, we recall the definition of the Hermite polynomials $H_k$ defined by
\begin{equation}\label{def:Hermite}
	H_k(x) = (-1)^k e^{x^2/2}\frac{\partial^k}{\partial x^k} \big(e^{-x^2/2}\big),
\end{equation}
which have in particular the property that $\esp_{X\sim\calN(\t,1)} H_k(X) = \t^k$.

Finally, we define our estimator of $\bF(\bt)$ as 
\begin{equation}\label{def:estimateur}
	\hat{\bF} = \sum_{i=1}^d \hat{F}(y_{1,i},y_{2,i})
\end{equation}
where 
\begin{equation}
	\hat{F}(u,v) = \sum_{l=0}^{L}\hat{P}_{K_l,M_l}(u) \fcar_{t_{l-1}<|v|\le t_l} + \hat{P}_{K_{L+1},M_{L+1}}(u) \fcar_{t_{L}<|v|},
\end{equation}
and for an arbitrary constant $c>0$,
\begin{equation}\label{parametres}
	\begin{cases}
		\ \hat{P}_{K,M}(u) = \sum_{k=1}^K a_{k,K,M} H_k(u), \phantom{\Big()} \\
		\ M_l = 2^{l}\sqrt{2\log(s^2/d)}, \quad K_l = \frac{c}{8} M_l^2, \phantom{\Big()} \\
		\ t_l = M_l/2,\, t_{-1}=0, \phantom{\Big()} \\
		\ L \text{ is the largest integer such that } 2^L < \sqrt{\log(d)/\log(s^2/d)}, \phantom{\Big()}\\ 
		\ M_{L+1} = \sqrt{2\log(d)}.
	\end{cases}
\end{equation}	

The next theorem is a slight modification of Theorem~1 in~\cite{CollierCommingesTsybakov2019}. It states the performance of our estimator in the case when the signal $\bt$ is not too sparse.
\begin{theorem}\label{theorem_upperbound_dense}
	 Assume that $2\sqrt{d} \le s\le d$ and that $F$ satisfies Assumptions~(A1-A2) with $\e_1+\e_2$ small enough. Then the estimator defined in~(\ref{def:estimateur}) with small enough $c$, depending on $\e_1$ and $\e_2$,  satisfies
	\begin{equation}
		\sup_{\bt\in\bT_{s,\sqrt{2\log(d)}}} \esp_\bt \big(\hat{\bF}-\bF(\bt)\big)^2 \le C_3\, s^2\max_{l=0,\ldots,L+1} \d^2_{K_l,M_l},
	\end{equation}
	where $C_3$ is some  positive constant, depending only on $C_1$ and $C_2$. 
\end{theorem}

Furthermore, in the case when no sparsity is assumed ($s=d$), we can derive a simpler statement for every segment $[-M,M]$ included in $[-\sqrt{\log(d)}, \sqrt{\log(d)}]$. To this end, we define the simplified estimator
\begin{equation}\label{def:estimateur2}
	\tilde{\bF} = \sum_{i=0}^d \hat{P}_{K,M}(y_i), \quad \hat{P}_{K,M}(u) = \sum_{k=0}^K a_{k,K,M} H_k(u),
\end{equation}
with $K=c\log(d)/\log(e\log(d)/M^2)$ for an arbitrary constant $c>0$. 

\begin{theorem}\label{theorem_upperbound_nonsparse}
	Assume that $0<M\le\sqrt{\log(d)}$, that for some constants $C'_1,C'_2>0$
	\begin{equation}
		\|F-F(0)\|_{\infty,[-M,M]} \le C'_1 d^{\e_1}, \quad \d_{K,M}^{-1} \le C'_2 d^{\e_2}, 
	\end{equation}
	and let $\tilde{\bF}$ be the estimator defined by~(\ref{def:estimateur2}). Then if $2\e_1+2\e_2<1$ and  if $c$ is chosen small enough, depending only on $\e_1+\e_2$,  then
	\begin{equation}
		\sup_{\bt \in \bT_{d,M}} \esp_{\bt} \big(\tilde{\bF}-\bF(\bt)\big)^2 \le C'_3 d^2\d_{K,M}^2,
	\end{equation}
	where $C'_3$ is some positive constant, depending only on $C_1'$ and $C_2'$. 
\end{theorem}

\section{Optimality results}\label{section:optimality}

The next theorem, which is a slight modification of Theorem~4 in \cite{CollierCommingesTsybakov2019}, states a lower bound on the minimax rate.

\begin{theorem}\label{theorem_sparse_lowerbound_2}
	Assume that $0<M\le\sqrt{\log(s^2/d)}$ and that for some constants $C''_1 , C''_2>0$ 
	\begin{equation}
		\|F-F(0)\|_{\infty,[-M,M]} \le C''_1 \Big(\frac{s^2}{d}\Big)^{\e_1}, \quad \d_{e^2\log(s^2/d)/\log(e\log(s^2/d)/M^2),M}^{-1} \le C''_2 \Big(\frac{s^2}{d}\Big)^{\e_2}, 
	\end{equation}
	and that Assumption~(A0) holds. Then there exists an absolute positive constant $C\ge 1$ such that if $s^2\ge C d$,  if $2\e_1+2\e_2<1$ and if $c$ is chosen small enough, depending on $\e_1+\e_2$, we have
	\begin{equation}
		\inf_{\hat{\bF}} \sup_{\bt\in\bT_{s,M}} \esp_{\bt} \big( \hat{\bF}-\bF(\bt)\big)^2 \ge C''_3 s^2\d_{e^2\log(s^2/d)/\log(e\log(s^2/d)/M^2),M}^2,
	\end{equation}
	for some positive  constant $C''_3$, depending only on $C''_1 $ and $C''_2$. 
\end{theorem}

But our estimation problem is more difficult than the problem where we know in advance that the nonzero coefficients belong to the first $k$ components of $\bt$ for $k\in\{s,\ldots,d\}$, and the last theorem gives lower bounds for these problems as well by replacing $d$ by $k$. This argument leads to the following corollary:

\begin{corollary}\label{corollary:lower}
	Let  Assumptions~(A0'-A1-A2) hold. Then there exist an  absolute positive constant $C\ge 1$ such that if $s^2\ge C d$,  if $2\e_1+2\e_2<1$ and if  $c$ is  chosen small enough, depending only on $\e_1+\e_2$,   we have

	\begin{equation}
		\inf_{\hat{\bF}} \sup_{\bt\in\bT_{s,\sqrt{\log(s)}}} \esp_{\bt} \big( \hat{\bF}-\bF(\bt)\big)^2 \ge C_4 s^2 \max_{s\le k\le d} \d_{e^2\log(s^2/k),\sqrt{\log(s^2/k)}}^2,
	\end{equation}
	for some positive constant $C_4$, depending only on $C_1 $, $C_2$.  \end{corollary}

Furthermore, the next theorem states that Assumption~(A3) is sufficient to prove that the upper bound from Theorem~\ref{theorem_upperbound_dense} matches with the lower bound from Corollary~\ref{corollary:lower}.

\begin{theorem}\label{theorem:minimax}
	 Let Assumptions~(A0'-A1-A2-A3) hold. Then there exist an  absolute positive constant $C\ge \sqrt{2}$ such that if $s^2\ge C d$,  if $\e_1,\e_2$  are small enough and $c$ is chosen small enough, depending only on $\e_1$ and $\e_2$, we have
	\begin{equation}
		C_5 \le \frac{\inf_{\hat{\bF}} \sup_{\bt\in\bT_{s,\sqrt{\log(d)}}} \esp_{\bt} \big( \hat{\bF}-\bF(\bt)\big)^2}{s^2 \max_{s\le k\le d} \d_{\log(s^2/k),\sqrt{\log(s^2/k)}}^2} \le C'_5,
	\end{equation}
	for some positive  constants $C_5,C'_5$, depending only on $C_1, C_2$ and $c$. 
\end{theorem}

This means in particular that for non-regular functionals satisfying the conditions (A0'-A1-A2-A3), the rate appearing in Theorem~\ref{theorem:minimax} must be the same as the rate found in~\cite{PolyanskiyWu2019}. More precisely, let us denote $\d_{\chi^2}(\frac{1}{\sqrt{d}})$ the following quantity 
\begin{align}
	\sup_{\pi_1,\pi_2} \Big\{ \Big| \int F(\t)\,\pi_1(d\t) - \int F(\t)\,\pi_2(d\t) \Big| \suchthat \chi^2(\prob_{\pi_1},\prob_{\pi_2})\le \frac1d, \esp_{\pi_i} \|\bt\|_0 \le s\Big\}
\end{align}
where the supremum is taken over all distribution probabilities on $[-M,M]^d$. According to Theorem 8 in~\cite{PolyanskiyWu2019}, if $\delta_{\chi^2}(\frac{1}{\sqrt{d}})\ge 31\frac{\|F\|_{\infty,[-M,M]}}{\sqrt{d}}$, then $\delta_{\chi^2}(\frac{1}{\sqrt{d}})$ is the minimax rate for estimating $\sum_{i=1}^d F(\t_i)$ over $\bT_{s,M}$, hence it is of the order of
\begin{equation}
	s^2 \max_{s\le k\le d} \d_{\log(s^2/k),\sqrt{\log(s^2/k)}}^2.
\end{equation}

Moreover, similar results as in~\cite{CaiLow2011,CollierCommingesTsybakov2019} (with bounded parameter space) can be easily deduced since for the function $x\to |x|^\g$, the approximation rate $\d_{K,M}$ is of the order of $(M/K)^\g$ (\cf~ for example Theorem~7.2.2 in~\cite{Timan1963}).

Finally, Assumption~(A2) is strongly related to the differentiability of the marginal functional $F$. Indeed, the following properties can be found in~\cite{Timan1963}, Sections~5.1.5 and~6.2.4: 
\begin{itemize}
	\item If $F$ has a bounded derivative of order $r$ on $[-1,1]$, then 
	\begin{equation}
		\forall n\ge1, \quad \d_{n,[-1,1]} \le \frac{C}{n^r},
	\end{equation}
	for some positive contant $C$.
	\item $F$ is infinitely derivable on $[a,b]$ if and only if for any $r>0$, 
	\begin{equation}
		n^r \d_{n,[a,b]} \to 0.
	\end{equation}
\end{itemize}
This suggests that many not infinitely differentiable functionals satisfy Assumption~(A2).

\section{Proof of theorems}\label{section:proof:theorem}

In the whole section, we denote by $A$ a positive constant  the value of which may vary from line to line. This constant only depends on $C_1$ and $C_2$ (Theorem~\ref{theorem_upperbound_dense}) and Theorem~\ref{theorem:minimax}), $C_1'$ and $C_2'$ (Theorem~\ref{theorem_upperbound_nonsparse}), $C_1^{''}$ and $C_2^{''}$  (Theorem~\ref{theorem_sparse_lowerbound_2}). Moreover, since 
\begin{equation}
	\esp_{\bt} \big( \hat{\bF}-\bF(\bt)\big)^2 = \esp_{\bt} \Big[ \big(\hat{\bF}-dF(0)\big)-\big(\bF(\bt)-dF(0)\big)\Big]^2,
\end{equation}
we can  assume without loss of generality that $F(0)=0$, which we do throughout this section.

\subsection{Proof of Theorem~\ref{theorem_upperbound_dense}}

Denote by $S$ the support of $\bt$. We start with a bias-variance decomposition
\begin{align}
	\big(\hat{\bF}-\bF(\bt)\big)^2 &\le 4\,\Big(\sum_{i\in S} \esp_\bt\hat{F}(y_{1,i},y_{2,i})-\sum_{i\in S} F(\t_i)\Big)^2 \\
	&+ 4\,\Big(\sum_{i\in S} \hat{F}(y_{1,i},y_{2,i})-\sum_{i\in S} \esp_\bt\hat{F}(y_{1,i},y_{2,i})\Big)^2 \\
	&+ 4\,\Big(\sum_{i\not\in S} \esp_\bt \hat{F}(y_{1,i},y_{2,i}) \Big)^2 \\
	&+ 4\,\Big(\sum_{i\not\in S} \hat{F}(y_{1,i},y_{2,i})-\sum_{i\not\in S} \esp_\bt\hat{F}(y_{1,i},y_{2,i})\Big)^2
\end{align}
leading to the bound
\begin{align}\label{proof_eq1}
	\esp_\bt(\hat{\bF}-\bF(\bt))^2 &\le 4s^2 \max_{i\in S} B_i^2 + 4s \max_{i\in S} V_i \\
	&+ 4d^2 \max_{i\not\in S} B_i^2 + 4d \max_{i\not\in S} V_i, \nonumber
\end{align}
where $B_i=\esp_\bt \hat{F}(y_{1,i},y_{2,i})- F(\t_i)$ is the bias of $\hat{F}(y_{1,i},y_{2,i})$ as an estimator of $F(\t_i)$ and $V_i=\var_\bt(\hat{F}(y_{1,i},y_{2,i}))$ is its variance. We now bound separately the four terms in~\eqref{proof_eq1}.  \\

$1^{\circ}.$ {\it Bias for $i\not\in S$.} If $i\not\in S$, then $B_i = 0$.

$2^{\circ}.$ {\it Variance for $i\not\in S$.} If $i\not\in S$, then using in particular Lemma~\ref{lemma_CaiLow1},
\begin{align}\label{proof_eq1a}
	V_i &\le \sum_{l=0}^{L+1} \esp \hat{P}^2_{K_l,M_l}(\xi) \, \prob(t_{l-1}<|\xi|), \quad \xi\sim\nzeroun, \\
	&\le A \sum_{l=0}^{L+1} \|F\|_{\infty,[-M_l,M_l]}^2 6^{K_l} e^{-t_{l-1}^2/2}.
\end{align}
For $l=0$, we have by Assumptions~(A1-A2), if $\e_1+\e_2<1/4$ and for $c$ small enough
\begin{equation}
	\|F\|_{\infty,[-M_0,M_0]}^2 6^{K_0} e^{-t_{-1}^2/2} \d^{-2}_{K_0,M_0} \le A \Big(\frac{s^2}{d}\Big)^{4\e_1 + 4\e_2 + c\log(6)/4} \le A \frac{s^2}{d}. 
\end{equation}
Then, if $l>0$,
\begin{equation}
	\|F\|_{\infty,[-M_l,M_l]}^2\, 6^{K_l} e^{-t_{l-1}^2/2} \d^{-2}_{K_l,M_l} \le A \Big(\frac{s^2}{d}\Big)^{4^l(4\e_1+4\e_2 + c\log(6)/4 - \frac1{16})}
\end{equation}
so that for small enough $c,\e_1$ and $\e_2$ and since $s^2\ge 4d$,
\begin{equation}
	dV_i \le A s^2\max_{l=0,\ldots,L} \d^2_{K_l,M_l}.
\end{equation}


\medskip

$3^{\circ}.$ {\it Bias for $i\in S$.} If $i\in S$, the bias has the form
\begin{align}
	B_i &= \sum_{l=0}^{L} \big\{ \esp \hat{P}_{K_l,M_l}(\xi) - F(\t_i) \big\}\,\prob(t_{l-1}<|\xi|\le t_l) \\ &+ \big\{ \esp \hat{P}_{K_{L+1},M_{L+1}}(\xi) - F(\t_i) \big\}\,\prob(t_{L}<|\xi|), \quad \xi\sim\calN(\t_i,1).
\end{align}
We will analyze this expression separately in different ranges of $|\t_i|$.

\smallskip

$3.1^{\circ}.$ {\it Case $0<|\t_i|\le2 t_0$. } In this case, we use the bound
\begin{align}
	|B_i| &\le \max_l \big|\esp \hat{P}_{K_l,M_l}(\xi)-F(\t_i)\big|, \quad \xi\sim\calN(\t_i,1).
\end{align}
Since $|\t_i|\le M_l$ for all~$l$, we have by the definition of $P_{K_l,M_l}$ and since $F(0)=0$,
\begin{align}\label{proof_eq3a}
	\big|\esp \hat{P}_{K_l,M_l}(\xi)-F(\t_i)\big| &\le \big| P_{K_l,M_l}(\t_i) - a_{0,K_l,M_l} - F(\t_i)\big| \\ &\le \big| P_{K_l,M_l}(\t_i) - F(\t_i)\big| + |F(0) - P_{K_l,M_l}(0)| \\ &\le 2\d_{K_l,M_l}, \phantom{\big|}
\end{align}
so that
\begin{equation}\label{proof_eq4}
	s^2 \max_{0<|\t_i|\le 2 t_0} B^2_i \le 4s^2\max_{l=0,\ldots,L+1} \d^2_{K_l,M_l}.
\end{equation}
\smallskip

$3.2^{\circ}.$ {\it Case $2 t_0<|\t_i|\le 2t_L$. }	
 Let $l_0\in\{0,\ldots,L-1\}$ be the integer such that $2t_{l_0} < |\t_i| \le 2t_{l_0+1}$.  We have
\begin{align}\label{proof_eq4a}
	|B_i| &\le \sum_{l=0}^{l_0} \big|\esp \hat{P}_{K_l,M_l}(\xi)-F(\t_i)\big|\cdot\prob(t_{l-1}<|\xi|\le t_l)  \\
	&+ \max_{l> l_0} \big|\esp \hat{P}_{K_l,M_l}(\xi)-F(\t_i)\big|, \quad \xi\sim\calN(\t_i,1)\nonumber
 \end{align}
The arguments in~\eqref{proof_eq3a} yield that 
\begin{equation}
	\max_{l> l_0}\big|\esp \hat{P}_{K_l,M_l}(\xi)-F(\t_i)\big| \le 4 \max_{l=0,\ldots,L+1} \d^2_{K_l,M_l}.
\end{equation}
Furthermore, using the triangular inequality,
\begin{align}
	 &\sum_{l=0}^{l_0} \big|\esp \hat{P}_{K_l,M_l}(\xi)-F(\t_i)\big|\cdot\prob(t_{l-1}<|\xi|\le t_l) \\ \le &\sum_{l=0}^{l_0} \big|\esp \hat{P}_{K_l,M_l}(\xi)\big|\cdot\prob(|\xi|\le t_l) + \sum_{l=0}^{l_0} \big|F(\t_i)\big|\cdot\prob(t_{l-1}<|\xi|\le t_l).
\end{align}
The first sum in the right-hand side can be bounded using Lemma~\ref{lemma_CaiLowamlior}, since
\begin{align}
	\big|\esp \hat{P}_{K_l,M_l}(\xi)\big|\,\prob(|\xi|\le t_l) \le A \|F\|_{\infty,[-M_l,M_l]}\, 3^{K_l} e^{c\t_i^2/16}\, \prob(|\xi|\le t_l), 
\end{align}		
so that, as $|\t_i|>2t_{l_0}\ge 2t_l$ for $l\le l_0$, 
\begin{align}
	\big|\esp \hat{P}_{K_l,M_l}(\xi)\big|\,\prob(|\xi|\le t_l) \d^{-1}_{K_l,M_l} &\le A 3^{K_l} e^{(c-2)\t_i^2/16} e^{(\e_1+\e_2)M_l^2} \\
	&\le A e^{(8\e_1+8\e_2+c\log(3)+(c-2)/2)t_l^2/2} \\ 
	&= A\Big(\frac{s^2}{d}\Big)^{2^{2l-2}(8\e_1+8\e_2+c\log(3)+(c-2)/2)}
\end{align}
Again, if $\e_1+\e_2<\frac{1}{8}$, choosing $c$ small enough yields that 
\begin{equation}
	\sum_{l=0}^{l_0} \big|\esp \hat{P}_{K_l,M_l}(\xi)\big|\,\prob(|\xi|\le t_l) \le A \max_{l=0,\ldots,L} \d_{K_l,M_l}.
\end{equation}
Moreover, similar arguments lead to the fact that if $4\e_1+\e_2 < \frac1{8}$
\begin{equation}
	\sum_{l=0}^{l_0} |F(\t_i)| \prob(t_{l-1}<|\xi|\le t_l) \le \|F\|_{\infty,[-M_{l_0+1},M_{l_0+1}]} \prob(|\xi|\le t_{l_0}) \le A \d_{K_{l_0},M_{l_0}},	
\end{equation}		
and we conclude that 
\begin{equation}\label{proof_eq5}
	s^2 \max_{2 t_0<|\t_i|\le2 t_L} B^2_i \le A s^2 \max_{l=0,\ldots,L+1} \d_{K_l,M_l}^2.
\end{equation}

$3.3^{\circ}.$ {\it Case $2t_L < |\t_i| \le \sqrt{2\log(d)}$. } Similar arguments as in the previous case yield that
\begin{equation}\label{proof_eq6}
	s^2 \max_{2t_L < |\t_i| \le \sqrt{2\log(d)}} B^2_i \le A s^2 \max_{l=0,\ldots,L+1} \d_{K_l,M_l}^2.
\end{equation}

\vspace{3mm}
\bigskip

$4^{\circ}.$ {\it Variance for $i\in S$.} We consider the same cases as in item $3^{\circ}$ above. In all cases, it suffices to bound the variance by the second-order moment, which grants that, for all $i\in S$, 
\begin{align} \label{rough}
	V_i \le \sum_{l=0}^{L} \esp \hat{P}^2_{K_l,M_l}(\xi) \,\prob(t_{l-1}<|\xi|\le t_l) + \esp \hat{P}^2_{K_{L+1},M_{L+1}}(\xi) \,\prob(t_{L}<|\xi|), \quad \xi\sim\calN(\t_i,1).
\end{align}

\medskip

$4.1^{\circ}.$ {\it Case $0<|\t_i|\le2 t_0$.}
In this case, we deduce from \eqref{rough} that
	\begin{equation}
	V_i \le \max_{l=0,\ldots,L+1} \esp \hat{P}^2_{K_l,M_l}(\xi), \quad \xi\sim\calN(\t_i,1).
\end{equation}
Lemma~\ref{lemma_CaiLow2} implies
\begin{align}
	V_i \le A \|F\|_{\infty,[-M_{L+1},M_{L+1}]}^2 12^{K_{L+1}} \le A d^{4\e_1+c\log(12)/4},
\end{align}
which, as $\sqrt{d}\le s$, is sufficient to conclude that
\begin{equation}\label{proof_eq7}
	s \max_{0<|\t_i|\le 2 t_0} V_i \le A s^2 \max_{l=0,\ldots,L+1} \d_{K_l,M_l}^2,
\end{equation}	
for $c, \e_1, \e_2$ small enough.
\smallskip

$4.2^{\circ}.$ {\it Case $2 t_0<|\t_i|\le2 t_L$.}	
 As in item $3.2^{\circ}$ above, we denote by $l_0\in\{0,\ldots,L-1\}$ the integer such that $2t_{l_0} < |\t_i| \le 2t_{l_0+1}$. We deduce from \eqref{rough} that
\begin{align}
	V_i &\le (l_0+1)\max_{l=0,\ldots,l_0} \esp \hat{P}^2_{K_l,M_l}(\xi) \, \prob(|\xi|\le t_{l_0}) + \max_{l=l_0+1,\ldots,L+1} \esp \hat{P}^2_{K_l,M_l}(\xi), \quad \xi\sim\calN(\t_i,1).
\end{align}
The last term on the right hand side is controlled as in item~$4.1^{\circ}$. For the first term, we find using Lemma~\ref{lemma_CaiLowamlior} that, for $\xi\sim\calN(\t_i,1)$,
\begin{align}\label{gg}
	\max_{l=0,\ldots,l_0} \esp \hat{P}^2_{K_l,M_l}(\xi) \, \prob(|\xi|\le t_{l_0}) &\le A \|F\|_{\infty,[-M_{l_0},M_{l_0}]}^2\, 6^{K_{l_0}} e^{\frac{c\log(1+8/c)}{8}\t_i^2}\, e^{-\t_i^2/8} \\
	&\le A  e^{(8\e_1+\frac{c\log 6}2+\frac{c\log(1+8/c)}2-\frac12)t^2_{l_0}}. 
\end{align}
Choosing $c,\e_1,\e_2$ small enough allows us to obtain the desired bound
\begin{equation}\label{proof_eq8}
	s \max_{2 t_0<|\t_i|\le2 t_L} V_i \le A s^2 \max_{l=0,\ldots,L+1} \d_{K_l,M_l}^2.
\end{equation}	

$4.3^{\circ}.$ {\it Case $2t_L < |\t_i| \le \sqrt{2\log(d)}$. } Similar arguments as in the previous case yield that
	\begin{equation}\label{proof_eq8}
	s \max_{2t_L < |\t_i| \le \sqrt{2\log(d)}} V_i \le A s^2 \max_{l=0,\ldots,L+1} \d_{K_l,M_l}^2.
\end{equation}	

\medskip	

The result of the theorem follows.

\subsection{Proof of Theorem~\ref{theorem_upperbound_nonsparse}}

By construction, we have
\begin{equation}
	\esp_{\bt} \big(\hat{\bF}-\bF(\bt)\big)^2 \le d^2\d_{K,M}^2 + \var\big(\hat{\bF}\big).
\end{equation}
To bound the variance, we write 
\begin{equation}
	\hat{\bF} = \sum_{k=0}^K a_{k,K,M} S_k, \quad S_k = \sum_{i=1}^d H_k(y_i),
\end{equation}
so that
\begin{equation}
	\var(\hat{\bF}) \le \Big( \sum_{k=0}^K |a_{k,K,M}| \sqrt{\var(S_k)} \Big)^2,
\end{equation}	 
since for any random variables $X_1,\ldots, X_n$, we have
\begin{equation}
	\esp\Big(\sum_{i=1}^n X_i\Big)^2 \le \Big( \sum_{i=1}^n \sqrt{\esp(X_i^2)} \Big)^2.
\end{equation}	
Furthermore, by Lemmas~\ref{lemma_coefficients} and~\ref{lemma_hermite2},
\begin{align}
	\sum_{k=0}^K |a_{k,K,M}| \sqrt{\var(S_k)} &\le A \sqrt{d} \|F\|_{\infty,[-M,M]}\, K(1+\sqrt2)^K \Big(1+\frac{K}{M^2}\Big)^{K/2}.
\end{align}		
Using the definition of $K$, we have
\begin{align}
	K\log(1+K/M^2) \le A c\log(d),
\end{align}
hence, taking $c$ small enough implies that 
\begin{align}
	\var(\hat{\bF}) \d^{-2}_{K,M} &\le A K^2(1+\sqrt2)^{2K} d^{2\e_1+2\e_2+1+Ac} \le A d^2.
\end{align}
The result follows.

\subsection{Proof of Theorem~\ref{theorem_sparse_lowerbound_2}}

\textbf{Preliminary:} By Markov's inequality, we have for every $K>0$
\begin{equation}\label{theorem_sparse_lowerbound_2:1}
	\inf_{\hat{\bF}} \sup_{\bt\in\bT} \esp_{\bt} \big( \hat{\bF}-\bF(\bt)\big)^2 \ge \frac{s^2\d_{K,M}^2}{4} \inf_{\hat{\bF}} \sup_{\bt\in\bT} \prob_{\bt} \Big( |\hat{\bF}-\bF(\bt)| \ge s\d_{K,M}/2 \Big),
\end{equation}
and Theorem~2.15 in~\cite{Tsybakov2009} implies that for any prior measures $\bar{\mu}_0$ and $\bar{\mu}_1$ concentrated on~$\bT$ 
\begin{equation}\label{theorem_sparse_lowerbound_2:2}	
	\inf_{\hat{\bF}} \sup_{\bt\in\bT} \prob_{\bt} \Big( |\hat{\bF}-\bF(\bt)| \ge \frac{m_1-m_0}{4} \Big) \ge \frac{1-V}{2}
\end{equation}	
with
\begin{equation}\label{egaliteV}
 	V = \mathrm{TV}(\bar{\prob}_0,\bar{\prob}_1) + \bar{\mu}_0\big(\bF(\bt)\ge m_0+3v_0 \big) + \bar{\mu}_1\big( \bF(\bt)\le \frac{m_0+m_1}{2}+3v_0 \big),
\end{equation}
where $\mathrm{TV}$ denotes the total-variation distance, and for $i=0,1$, $\bar{\prob}_i$ is defined for every measurable set by
\begin{equation}
	\bar{\prob}_i(A) = \int_{\RR^d} \prob_{\bt} (A)\,\bar{\mu}_i(d\bt)
\end{equation}
and $m_0,m_1,v_0$ are to be chosen later. 

\textbf{Construction of the prior measures:} First we choose
\begin{equation}
	K = \frac{e^2\log(s^2/d)}{\log(e\log(s^2/d)/M^2)},
\end{equation}
and we denote $\mu_i$ for $i\in\{0,1\}$  the distribution of the random vector $\bt\in\RR^d$ with independent components distributed as $\e\eta_i$, where $\e$ and $\eta_i$ are independent, $\e\sim\mathcal{B}\big(s/(2d)\big)$ and $\eta_i$ is distributed according to $\nu_i$ defined in~Lemma~\ref{lemma_mesure}. Then, we define probabilities $\prob_0$ and $\prob_1$ by
\begin{equation}
	\prob_i(A) = \int_{\RR^d} \prob_{\bt}(A)\,\mu_i(d\bt),
\end{equation}
for all measurable sets $A$. The densities of these probabilities with respect to the Lebesgue measure on $\RR^d$ are given by
\begin{equation}
	f_i(x) = \prod_{i=1}^d g_i(x_i), 
\end{equation}
where 
\begin{equation}
	g_i(x) = \frac{s}{2d} \phi_i(x) + \Big(1-\frac{s}{2d}\Big) \phi(x), 
\end{equation}
and
\begin{equation}
	\phi_i(x) = \int_{\RR} \phi(x-t)\,\nu_i(dt), \quad \phi(x) = \frac1{\sqrt{2\pi}} e^{-x^2/2}.
\end{equation}
But as the $\mu_i$'s are not supported on $\bT$, we define counterparts $\bar{\mu}_i$'s by
\begin{equation}
	\bar{\mu}_i(A) = \frac{\mu_i(A\cap \bT)}{\mu_i(\bT)}.
\end{equation}
Finally, we denote
\begin{equation}
	m_i = \int_{\RR^d} \bF(\bt)\,\mu_i(d\bt), \quad v_i^2 = \int_{\RR^d} (\bF(\bt)-m_i)^2\,\mu_i(d\bt).
\end{equation}
	
\textbf{Bounding the probabilities in~(\ref{egaliteV}):}
According to Lemma~\ref{lemma_mesure}, we have
\begin{equation}\label{m1-m0}
	m_1 - m_0 = d \times \frac{s}{2d} \times \Big( \int_{-M}^M F(t)\,\nu_1(dt) - \int_{-M}^M F(t)\,\nu_0(dt) \Big) = s\d_{K,M}.	
\end{equation}
Using Lemma~9 in~\cite{CollierCommingesTsybakov2019} and Chebyshev-Cantelli's inequality, we have for $d$ large enough
\begin{align}\label{theorem_sparse_lowerbound_2:3}
	\bar{\mu}_0\big(\bF(\bt)\ge m_0+3v_0 \big) &\le \mu_0\big(\bF(\bt)\ge m_0+3v_0 \big) + e^{-s/16}\\
	&\le \frac{v_0^2}{v_0^2+(3v_0)^2} + e^{-s/16} < \frac15.
\end{align}
Now, we notice that for $i\in\{0,1\}$, we have
\begin{equation}
	v_i^2 \le d \|F\|_{\infty,[-M,M]}^2,
\end{equation}
so that for $C$ large enough,
\begin{align}
	\frac{m_0+m_1}{2}+3v_0 - m_1 \le 3\sqrt{d}\|F\|_{\infty,[-M,M]} - \frac{s\d_{K,M}}2 \le -\frac{s\d_{K,M}}3,
\end{align}
since the assumptions of the theorem imply that
\begin{equation}
	\frac{\sqrt{d}}{s\d_{K,M}} \|F\|_{\infty,[-M,M]} \le A \Big(\frac{s^2}{d}\Big)^{\e_1+\e_2-1/2}.
\end{equation}
Consequently,
\begin{align}
	 \bar{\mu}_1\big( \bF(\bt)\le \frac{m_0+m_1}{2}+3v_0 \big) &\le  \mu_1\big( \bF(\bt) - m_1 \le -\frac{s\d_{K,M}}{3} \big) + e^{-s/16} \\
	 &\le \frac{9v_1^2}{9v_1^2+s^2\d^2_{K,M}} + e^{-s/16},
\end{align}
by Chebyshev-Cantelli's inequality, and the last quantity is smaller than
\begin{equation}
	\frac{9d \|F\|_{\infty,[-M,M]}^2}{9d \|F\|_{\infty,[-M,M]}^2+s^2\d^2_{K,M}} + e^{-s/16}.
\end{equation}
Finally, we use again the fact that $d \|F\|_{\infty,[-M,M]}^2/(s^2\d^2_{K,M}) \le A (d/s^2)^{1-2\e_1-2\e_2}$ with $s^2/d>C$, so that for $C$ large enough, 
\begin{equation}\label{theorem_sparse_lowerbound_2:5}	
\bar{\mu}_1\big( \bF(\bt)\le \frac{m_0+m_1}{2}+3v_0 \big) < \frac15.
\end{equation}

\textbf{Bounding the total-variation distance in~(\ref{egaliteV}):}
We can upper bound the total-variation distance as follows:
\begin{align}
	\mathrm{TV}(\bar{\prob}_0,\bar{\prob}_1) &\le \mathrm{TV}(\bar{\prob}_0,\prob_0) + \mathrm{TV}(\prob_0,\prob_1) + \mathrm{TV}(\prob_1,\bar{\prob}_1) \\
	&\le \sqrt{\chi^2(\prob_0,\prob_1)/2} + \mu_0(\bT^\complement) + \mu_1(\bT^\complement),
\end{align}
where $\bT^\complement$ denotes the complement of $\bT$. 
As before, 
\begin{equation}\label{theorem_sparse_lowerbound_2:6}
	\mu_i(\bT^\complement) \le \prob\Big( \mathcal{B}\big(d,\frac{s}{2d}\big) > s \Big) \le e^{-\frac{s}{16}}.
\end{equation}
Furthermore, since the $\prob_i$'s are product measures, we have
\begin{equation}
	\chi^2(\prob_0,\prob_1) = \Big(1 + \int \frac{(g_1-g_0)^2}{g_0} \Big)^d - 1,
\end{equation}
and by the definition of $g_0,g_1$, 
\begin{equation}
	\int \frac{(g_1-g_0)^2}{g_0} \le \frac{1}{1-\frac{s}{2d}} \Big(\frac{s}{2d}\Big)^2 \int \frac{(\phi_1-\phi_0)^2}{\phi} \le \frac{s^2}{2d^2} \int \frac{(\phi_1-\phi_0)^2}{\phi}.
\end{equation}
Then
\begin{align}
	\int \frac{(\phi_1-\phi_0)^2}{\phi} &= \int e^{\t\t'} \nu_1(d\t)\nu_1(d\t') + \int e^{\t\t'} \nu_0(d\t)\nu_0(d\t') - 2 \int e^{\t\t'} \nu_0(d\t)\nu_1(d\t') \phantom{\sum_{k\ge K+1}}\\
	&= \sum_{k\ge0} \frac1{k!} \Big( \int t^k\nu_1(dt) - \int t^k\nu_0(dt) \Big)^2 \\
	&\le 4\sum_{k\ge K+1} \frac{M^{2k}}{k!},
\end{align}
and the choice of $K$ along with the condition on $M$ imply that $eM^2/K\le1/e$, so that
\begin{equation}
	\int \frac{(\phi_1-\phi_0)^2}{\phi} \le 4\sum_{k\ge K+1} \Big(\frac{eM^2}{k}\Big)^k \le 4\Big(\frac{eM^2}{K}\Big)^K.
\end{equation}
Coming back to the $\chi^2$-distance and using the fact that $1+x\le e^x$ for every $x\in\RR$, we get
\begin{equation}
	\chi^2(\prob_0,\prob_1) \le \exp \Big[  \frac{2s^2}{d}\Big(\frac{eM^2}{K}\Big)^K \Big] - 1.
\end{equation}
Finally, we compute
\begin{align}
	K \log\Big(\frac{eM^2}{K}\Big) =  - e^2\log(s^2/d)\times g\Big(e\log(s^2/d)/M^2\Big),
\end{align}
where 
\begin{equation}
	g(x) = \frac{\log\Big(\frac{x}{\log(x)}\Big)}{\log(x)},
\end{equation}
and it can be shown that $g>0.5$, so that $\chi^2(\prob_0,\prob_1) \le e^{2(d/s^2)^{e^2/2-1}}-1$.  This inequality, combined with \eqref{theorem_sparse_lowerbound_2:6}, yields
\begin{equation} \label{theorem_sparse_lowerbound_2:7} \mathrm{TV}(\bar{\prob}_0,\bar{\prob}_1) <3/5
\end{equation}
 if $C$ and $d$ are large enough. 
 
The proof is completed by gathering  \eqref{theorem_sparse_lowerbound_2:1}, \eqref{m1-m0}, \eqref{theorem_sparse_lowerbound_2:2}, \eqref{egaliteV}, \eqref{theorem_sparse_lowerbound_2:3},   \eqref{theorem_sparse_lowerbound_2:5} and \eqref{theorem_sparse_lowerbound_2:7}. 
 
\subsection{Proof of Theorem~\ref{theorem:minimax}}

If $l \in \{0,\ldots,L+1\}$, then by definition of $K_l$ in~\eqref{parametres}, we have\begin{equation}
	K_l\le \frac{c}4 \log(d) \le \frac{c}4\log(s^2/C) \le \log(s)
\end{equation}
for $c$ small enough. Besides, if $l_0=\big\lfloor \frac{\log_2(4/c)}{2}\big\rfloor +1$, where $\lfloor \cdot \rfloor$ denotes the integer part, then 
\begin{equation}
	\forall l\ge l_0, \quad K_l\ge \log(s^2/d).
\end{equation}
On the other hand, when $k\in\{s,\ldots,d\}$, the quantity $\log(s^2/k)$ ranges from $\log(s^2/d)$ to $\log(s)$ and the consecutive differences satisfy
\begin{equation}
	\log\big(s^2/k\big)-\log\big(s^2/(k+1)\big) = \log(1+1/k) \in [0,1],
\end{equation}
so that for every $l \in \{l_0,\ldots,L+1\}$, there exists an integer $k_l\in\{s,\ldots,d\}$ such that
\begin{equation}
	|K_l-\log(s^2/k_l)|\le 1.
\end{equation}
Now note that $\log(s^2/k_l)\ge \log(C)$, which yields that,  for every $l \in \{l_0,\ldots,L+1\}$,
\begin{equation}
	\frac{K_l}{\log(s^2/k_l)} = 1 + \frac{K_l-\log(s^2/k_l)}{\log(s^2/k_l)} \in \Big[1-\frac1{\log(C)}, 1+\frac1{\log(C)}\Big].
\end{equation}
But for $l\in \{ 0,\ldots, l_0-1\}$, we have 
\begin{equation}
	1 \le \frac{K_l}{K_0} \le \frac{4}{c},
\end{equation}
so that the last two displays, combined with Assumption~(A3), entail that
\begin{equation}
	\max_{l=0,\ldots,L+1} \d^2_{K_l,M_l} \le A \max_{l=0,\ldots,L+1} \d_{\log(s^2/k_l),\sqrt{\log(s^2/k_l)}}^2 \le A \max_{k=s,\ldots,d} \d_{\log(s^2/k),\sqrt{\log(s^2/k)}}^2.
\end{equation}
Finally, we conclude by Assumption~(A3) again, since
\begin{equation}
	\max_{k=s,\ldots,d} \d_{\log(s^2/k),\sqrt{\log(s^2/k)}}^2 \le A \max_{k=s,\ldots,d} \d_{e^2\log(s^2/k),\sqrt{\log(s^2/k)}}^2.
\end{equation}
	
\section{Lemmas}\label{section:lemma}

In the whole section, we denote by $A$ an absolute positive constant that precise value may vary from line to line. 

The following lemma is a direct consequence of Proposition~2 in~\cite{CollierCommingesTsybakov2019}.
\begin{lemma}\label{lemma_coefficients}
	Let $P_{K,M}$ be the polynomial defined in~(\ref{def:PKM}). Then the coefficients $a_{k,K,M}$ in~(\ref{def:coefficients}) satisfy
	\begin{equation}
		|a_{k,K,M}| \le A \|F\|_{\infty,[-M,M]} M^{-k} (1+\sqrt2)^{K}, \quad k=0,\dots,K.
	\end{equation}
\end{lemma}

The following lemma is a slight modification of Lemma~1 in~\cite{CaiLow2011}:

\begin{lemma}\label{lemma_mesure}
	Assume that $F$ is continuous on $[-M,M]$, then for every positive integer~$K$, if $\d_{K,M}>0$, there exist measures $\nu_0, \nu_1$ on $[-M,M]$ such that
	\begin{equation}
		\begin{cases}
			\ \int t^l \nu_0(dt)= \int t^l \nu_1(dt), \quad l=0,\ldots,K \\
			\ \int F(t) \nu_0(dt) - \int F(t) \nu_1(dt) = 2\d_{K,M}.
		\end{cases}
	\end{equation}
\end{lemma}
	
\begin{proof}
	Denote $\mathcal{C}$ the set of continuous functions on $[-M,M]$ equipped with the uniform norm, and $\mathcal{F}_k$ be the linear space spanned by $\mathcal{P}_K$ (the set of polynomials of degree smaller than $K$) and $F$. Note that $F$ does not belong to $\mathcal{P}_K$, since by assumption, $\d_{K,M}>0$. Then every element $g$ of $\mathcal{F}_K$ can be represented as $g=cF+P$, where $P\in\mathcal{P}_K$ and $c\in\RR$. Then we can define the linear functional $T$ on $\mathcal{F}_K$ by $T(g)=c\d_{K,M}$. We then compute the norm of $T$ defined as
	\begin{equation}
		\|T\| = \sup\{ T(g) \suchthat \|g\|_\infty = 1 \}.
	\end{equation}
	Now, every $g\in\mathcal{F}_K$ satisfying $\|g\|_\infty=1$ can be written as
	\begin{equation}
		g = \frac{cF+P}{\|cF+P\|_\infty}, \quad P\in\mathcal{P}_K,
	\end{equation}
	so that
	\begin{equation}
		\|T\| = \sup_{c,P} \frac{c\d_{K,M}}{\|cF+P\|_\infty} = \sup_P \frac{\d_{K,M}}{\|F-P\|_\infty} = 1
	\end{equation}
	 by definition of $\d_{K,M}$. Then, using Hahn-Banach and Riesz representation theorems, we can extend $T$ on $\mathcal{C}$ without changing its norm, and represent this extension $\tilde{T}$ as
	\begin{equation}
		\tilde{T}(g) =  \int_{-M}^M g(t)\,\tau(dt),
	\end{equation}
	where $\tau$ is a signed measure with total variation $1$. Then, using Jordan decomposition, we can write $\tau$ as a difference of two positive measures
	\begin{equation}
		\tau=\tau^+-\tau^-.
	\end{equation} 
	Denoting $\nu_0=2\tau^+$ and $\nu_1=2\tau^-$, which are probability measures since $2\tau$ has total variation $2$ and $\int_{-M}^M \tau(dt)=0$, the last properties of the lemma follow from the properties of $\tau$.
\end{proof}

The proof of the next lemma can be found in~\cite{CaiLow2011}.

\begin{lemma}\label{lemma_hermite2}
	Let $\t\in \RR$ and $X\sim \calN(\t,1)$. For any $k\in \mathbb N$, the $k$-th Hermite polynomial satisfies
	\begin{align}
				\esp H_k(X) &= \t^k, \\ 
			\esp H_k^2(X) &\le \big(k+\t^2\big)^k. 
	\end{align}
\end{lemma}

\begin{lemma}\label{lemma_CaiLow1}
	Let $\hat{P}_{K,M}$ be defined in~(\ref{parametres}) with $K\le M^2$. If $\xi\sim\calN(0,1)$, then
	\begin{equation}
	\esp \hat{P}_{K,M}^2(\xi) \le A \|F\|^2_{\infty,[-M,M]}\, 6^{K}.
	\end{equation}
\end{lemma}

\begin{proof}
	Recall that, for the Hermite polynomials, $\esp (H_k(\xi)H_j(\xi))=0$ if $k\ne j$ and $\xi \sim\calN(0,1)$.
	Using this fact and then Lemmas~\ref{lemma_coefficients} and~\ref{lemma_hermite2} we obtain
	\begin{align}
		\esp \hat{P}_{K,M}^2(\xi) = \sum_{k=1}^K a^2_{k,K,M} \esp H^2_{k}(\xi) \le A \|F\|^2_{\infty,[-M,M]}\, (1+\sqrt2)^{2K} \sum_{k=1}^K (k/M^2)^{k}.
	\end{align}
	Moreover, since $K/M^2 \le 1$, we have $\sum_{k=1}^K (k/M^2)^{k} \le K$. The result follows.
\end{proof}
	
\begin{lemma}\label{lemma_CaiLow2}
	Let $\hat{P}_{K,M}$ be defined in~(\ref{parametres}) with parameters $K=cM^2/8$ and $c\le8$. If $\xi\sim\calN(\t,1)$ with $|\t|\le M$, then
	\begin{align}
		 \esp \hat{P}^2_{K,M}(X) \le A \|F\|_{\infty,[-M,M]}^2\, 12^{K}.
	\end{align}
\end{lemma}

\begin{proof}
	We use the bound
	\begin{equation}\label{secmom}
		\esp \hat{P}^2_{K,M}(\xi) \le \bigg(\sum_{k=1}^K |a_{k,K,M}| \sqrt{\esp H_{k}^2(\xi)} \bigg)^2.
	\end{equation}
	Thus Lemma~\ref{lemma_hermite2} in particular and the fact that $K\le M^2$ imply that
	\begin{align}
		\esp \hat{P}^2_{K,M}(\xi) \le A \|F\|_{\infty,[-M,M]}^2 (1+\sqrt2)^{2K} \bigg(\sum_{k=1}^K M^{-k} 2^{k/2} M^{k} \bigg)^2 \le A \|F\|_{\infty,[-M,M]}^2\,12^{K}.
	\end{align}
\end{proof}

\begin{lemma}\label{lemma_CaiLowamlior}
	Let $\hat{P}_{K,M}$ be defined in~(\ref{parametres}) with $K=cM^2/8$ and $c\le 8$. If $\xi\sim\calN(\t,1)$ with $|\t|>M$, then
	\begin{align}
		\big| \esp \hat{P}_{K,M}(\xi) \big| &\le A \|F\|_{\infty,[-M,M]}\, 3^{K} e^{c\t^2/16}, \\
		\esp \hat{P}^2_{K,M}(\xi) &\le A \|F\|_{\infty,[-M,M]}^2\, 6^{K} e^{\frac{c\log(1+8/c)}{8}\t^2}.
	\end{align}
\end{lemma}

\begin{proof}
	To prove the first inequality of the lemma, we use Lemma~\ref{lemma_coefficients} to obtain
	\begin{align}
		\big| \esp \hat{P}_{K,M}(\xi) \big| \le A \|F\|_{\infty,[-M,M]} K (1+\sqrt2)^{K} \Big(\frac{|\t|}{M}\Big)^K,
	\end{align}
	and the result follows from 
	\begin{equation}
		K\log(|\t|/M)= \frac{cM^2}8 \log(|\t|/M) \le c\t^2/16.
	\end{equation}		
	
	We now prove the  second  inequality of the lemma. Using \eqref{secmom} and then Lemmas~\ref{lemma_coefficients} and~\ref{lemma_hermite2} we get
	\begin{align}
		\esp \hat{P}^2_{K,M}(\xi) \le A \|F\|_{\infty,[-M,M]}^2\, (1+\sqrt2)^{2K} \Big( \sum_{k=1}^K M^{-k} (k+\t^2)^{k/2} \Big)^2.
	\end{align}
	But as $\frac{\t^2}{k} \ge \frac{M^2}{K} = \frac8{c}\ge 1$, we can use the fact that the function $x\to x^{-1}\log(1+x)$ is decreasing on $\RR_+^*$ to obtain that
	\begin{equation}
		k \log\Big(1+\frac{\t^2}{k}\Big)\le \frac{c\t^2\log(1+8/c)}{8}.	
	\end{equation}
	Therefore,
	\begin{align}
		\esp \hat{P}^2_{K,M}(\xi) &\le A \|F\|_{\infty,[-M,M]}^2\, (1+\sqrt2)^{2K} e^{\frac{c\log(1+8/c)}{8}\t^2}  \bigg( \sum_{k=1}^K (k/M^2)^{k/2} \bigg)^2.
	\end{align}
	Finally, the result follows since $K\le M^2$.
\end{proof}


\acks{We thank A.B. Tsybakov for fruitful discussions during the redaction of this paper. Olivier Collier’s research has been conducted as part of the project Labex MME-DII (ANR11- LBX-0023-01). }




\vskip 0.2in
\bibliography{ref}

\end{document}